\documentclass{amsart}
\usepackage{amssymb,amsmath,amsthm}

\theoremstyle{plain}
\newtheorem{theorem}{Theorem}[section]

\newtheorem{corollary}[theorem]{Corollary}

\theoremstyle{definition}

\theoremstyle{remark}

\newcommand{\RR}{{\mathbb R}}
\newcommand{\TT}{{\mathbb T}}
\newcommand{\DD}{{\mathbb D}}
\newcommand{\cD}{{\mathcal D}}

\begin{document}

\title[Carleson measures]{One-box conditions for Carleson measures for the Dirichlet space}

\author[El-Fallah]{Omar El-Fallah}
\address{Laboratoire Analyse et Applications (CNRST URAC03)
\\Universit\'e Mohamed V\\B.~P.~1014 Rabat\\Morocco}
\email{elfallah@fsr.ac.ma}

\author[Kellay]{Karim Kellay}
\address{IMB\\
Universit\'e Bordeaux 1\\
351 cours de la Lib\'eration\\
F-33405 Talence cedex\\France}
\email{karim.kellay@math.u-bordeaux1.fr}

\author[Mashreghi]{Javad Mashreghi}
\address{D\'{e}partement de math\'{e}matiques et de statistique
\\Universit\'{e} Laval\\Qu\'{e}bec (QC) \\ Can\-a\-da G1V 0A6}
\email{javad.mashreghi@mat.ulaval.ca}

\author[Ransford]{Thomas Ransford}
\address{D\'{e}partement de math\'{e}matiques et de statistique
\\Universit\'{e} Laval\\Qu\'{e}bec (QC) \\ Can\-a\-da G1V 0A6}
\email{ransford@mat.ulaval.ca}

\thanks{First author supported by Acad\'emie Hassan II des sciences et techniques. Second author supported by PICS-CNRS.
Third author supported by NSERC.
Fourth author  supported by NSERC and the Canada research chairs program.} 


\date{February 15, 2013}

\commby{??}

\begin{abstract}
We give a simple proof of the fact that a finite measure $\mu$
on the unit disk is a Carleson measure for the Dirichlet space
if it satisfies the Carleson one-box condition $\mu(S(I))=O(\phi(|I|))$,
where $\phi:(0,2\pi]\to(0,\infty)$ is an increasing function such that
$\int_0^{2\pi}(\phi(x)/x)\,dx<\infty$. We further show that the
integral condition on $\phi$ is sharp.
\end{abstract}

\maketitle

\section{Introduction}

Let $(F,\|\cdot\|_F)$ be a Banach space of measurable
functions defined on a measurable space~$X$. 
A \emph{Carleson measure} for $F$ is
a positive measure $\mu$ on $X$ such that $F$
embeds continuously into $L^2(\mu)$.
In other words, $\mu$ is a Carleson measure for $F$ if
there exists a constant $C$ such that
\[
\int_X|f(x)|^2\,d\mu(x)\le C\|f\|_F^2
\qquad(f\in F).
\]

Carleson measures  were introduced
by Carleson  \cite{Ca62} in his solution to the corona problem.
He considered the case where $X=\DD$, the unit disk, and $F=H^p$,  the Hardy spaces.
In this case he obtained a rather simple geometric characterization of these measures.
Given an arc $I$ in the unit circle, let us write
$|I|$ for its arclength, 
and $S(I)$ for the associated Carleson box, defined by
\[
S(I):=\{re^{i\theta}:1-|I|<r<1,~e^{i\theta}\in I\}.
\]
Then, for each $p\in[1,\infty]$, a finite measure $\mu$ on $\DD$ is a Carleson measure
for $H^p$ if and only if
\begin{equation}\label{E:CarlHp}
\mu(S(I))=O(|I|)
\qquad(|I|\to0).
\end{equation}

In this article, 
we are interested in the case where $X=\DD$ and $F=\cD$,
the Dirichlet space. By definition, $\cD$ is the space 
of  functions $f$ holomorphic in  $\DD$ 
whose Dirichlet integral is finite, i.e.,
\[
\cD(f):=\frac{1}{\pi}\int_\DD |f'(z)|^2\,dA(z)<\infty.
\]
We can make $\cD$ into a Hilbert space by defining 
\[
\|f\|_\cD^2:=|f(0)|^2+\cD(f)
\qquad(f\in\cD).
\]

Carleson measures for $\cD$ arise in several contexts,
notably in characterizing multipliers \cite{St80}
and interpolation sequences \cite{Bi94,Bo05,MS89,Se04}.
The problem of characterizing Carleson measures themselves
has been studied by a number of authors over the years.
The analogue of \eqref{E:CarlHp} is the condition
\begin{equation}\label{E:CarlD}
\mu(S(I))=O\Bigl(\Bigl(\log\frac{1}{|I|}\Bigr)^{-1}\Bigl(\log\log\frac{1}{|I|}\Bigr)^{-\alpha}\Bigr)
\qquad(|I|\to0),
\end{equation}
which, for a finite measure $\mu$ on $\DD$, is known to be:
\begin{itemize}
\item necessary for $\mu$ to be Carleson for $\cD$ if $\alpha=0$;
\item sufficient for $\mu$ to be Carleson for $\cD$ if $\alpha>1$;
\item neither necessary nor sufficient if $0<\alpha\le 1$.
\end{itemize}
The gap between necessity and sufficiency means that one cannot
completely characterize Carleson measures for $\cD$  in terms of a one-box
condition like \eqref{E:CarlHp} or \eqref{E:CarlD}. 
There do exist complete characterizations, 
but they are of a more complicated nature, 
see for example  \cite{ARS02,ARS08,ARSW11,KS88,St80}.

We shall discuss 
the necessity part of \eqref{E:CarlD} briefly at the end of the paper.
However, our main interest is the sufficiency part of \eqref{E:CarlD},
which is a consequence of the following more precise  results.

\begin{theorem}\label{T:Wynn}
Let $\mu$ be a finite positive Borel measure on $\DD$ satisfying
\begin{equation}\label{E:Wynn}
\mu(S(I))=O(\phi(|I|))
\qquad(|I|\to0),
\end{equation}
where $\phi:(0,2\pi]\to(0,\infty)$ is an increasing function such that 
\begin{equation}\label{E:phicond}
\int_0^{2\pi}\frac{\phi(x)}{x}\,dx<\infty.
\end{equation}
Then $\mu$ is a Carleson measure for $\cD$.
\end{theorem}

The condition \eqref{E:phicond} in Theorem~\ref{T:Wynn}
is sharp in the following sense.

\begin{theorem}\label{T:Wynnsharp}
Let $\phi:(0,2\pi]\to(0,\infty)$ be a continuous increasing function 
such that $\phi(x)/x$ is strictly decreasing
and 
\begin{equation}\label{E:Wynnsharp}
\int_0^{2\pi}\frac{\phi(x)}{x}\,dx=\infty.
\end{equation}
Then there exists a finite positive Borel measure
$\mu$ on $\DD$ that satisfies \eqref{E:Wynn} but is not a Carleson measure for $\cD$.
\end{theorem}

Theorem~\ref{T:Wynn} provides a justification of the sufficiency of \eqref{E:CarlD}
when $\alpha>1$, and Theorem~\ref{T:Wynnsharp} demonstrates its insufficiency
when $\alpha\le1$.

Both theorems were recently obtained by Wynn \cite{Wy11} 
under additional assumptions on $\phi$
(see Theorems~1.3 and~1.5 in that paper, as well as the discussion in \S4.2). 
Wynn's proofs are rather indirect, since they are a by-product
of his work on the so-called discrete Weiss conjecture. 
Our purpose is to give  simpler and more direct proofs.
We shall prove Theorems~\ref{T:Wynn} and \ref{T:Wynnsharp}
in \S\ref{S:pfthm1} and \S\ref{S:pfthm2} respectively,
and conclude in \S\ref{S:Carlnec} with some brief remarks about 
necessity.

\section{Proof of Theorem~\ref{T:Wynn}}\label{S:pfthm1}

Let us write $\langle\cdot,\cdot\rangle_\cD$ for the inner product on $\cD$.
Thus $\langle f,f\rangle_\cD=\|f\|_\cD^2$ for all $f\in\cD$. Also, let
\[
k(z,w)=k_w(z):=1+\log\Bigl(\frac{1}{1-\overline{w}z}\Bigr)
\qquad(z,w\in\DD).
\]
It is easy to verify that $k$ is a reproducing kernel for $\cD$,
in the sense that
\[
f(w)=\langle f,k_w\rangle_\cD 
\qquad(f\in\cD,~w\in\DD).
\]
We shall need the following dual formulation of the notion of Carleson measure. 
It is a special case of an abstract result of Arcozzi, Rochberg and Sawyer
\cite[Lemma~24]{ARS08}.

\begin{theorem}\label{T:Carlabs}
Let $\mu$ be a finite positive Borel measure on $\DD$.
Then
\begin{equation}\label{E:Carlabs}
\sup_{\|f\|_\cD\le1}\int_\DD |f(z)|^2\,d\mu(z)
=\sup_{\|g\|_{L^2(\mu)}\le1}\Bigl|\int_\DD\int_\DD 
k(w,z)g(z)\overline{g(w)}\,d\mu(z)\,d\mu(w)\Bigr|.
\end{equation}
\end{theorem}

\begin{corollary}\label{C:Carlabs}
A finite positive measure $\mu$ on $\DD$ is a Carleson measure for $\cD$ if and only if
\begin{equation}\label{E:CarlabsC}
\sup_{\|g\|_{L^2(\mu)}\le1}\int_\DD\int_\DD
\log\Bigl(\frac{2}{|1-\overline{w}z|}\Bigr)|g(z)||g(w)|\,d\mu(z)\,d\mu(w)<\infty.
\end{equation}
\end{corollary}

\begin{corollary}\label{C:Carllog}
A finite positive measure $\mu$ on $\DD$ is a Carleson measure for $\cD$ provided that
\begin{equation}\label{E:Carllog}
\sup_{w\in\DD}\int_\DD\log\Bigl(\frac{2}{|1-\overline{w}z|}\Bigr)\,d\mu(z)<\infty.
\end{equation}
\end{corollary}

\begin{proof}
It suffices to check that \eqref{E:Carllog} implies \eqref{E:CarlabsC}.
Let  $M$ be the supremum in \eqref{E:Carllog},
and write $L(w,z):=\log(2/|1-\overline{w}z|)$.
By the Cauchy--Schwarz inequality and the fact that $L(w,z)=L(z,w)$,
we have
\begin{align*}
\int_\DD\int_\DD L(w,z)|g(z)||g(w)|\,d\mu(z)\,d\mu(w)
&\le \int_\DD\int_\DD L(z,w)|g(w)|^2\,d\mu(z)\,d\mu(w)\\
&\le M\|g\|_{L^2(\mu)}^2.\qedhere
\end{align*}
\end{proof}

\begin{proof}[Proof of Theorem~\ref{T:Wynn}]
It suffices to show that \eqref{E:Wynn}
implies \eqref{E:Carllog}.
In establishing \eqref{E:Carllog},
we can restrict our attention to those $w$
with $1/2<|w|<1$, since the supremum over the remaining $w$
is clearly finite. For convenience, we extend the domain
of definition of $\phi$ to the whole of $\RR^+$ by setting
$\phi(t):=\phi(2\pi)$ for $t>2\pi$.

Fix $w$ with $1/2<|w|<1$. Using Fubini's theorem to integrate by parts, we have
\[
\int_\DD\log\Bigl(\frac{2}{|1-\overline{w}z|}\Bigr)\,d\mu(z)
=\int_{t=0}^2\mu\Bigl(\{z\in\DD:|1-\overline{w}z|\le t\}\Bigr)
\frac{1}{t}\,dt.
\]
Now $\{z\in\DD:|1-\overline{w}z|\le t\}=\{z\in\DD:|z-1/\overline{w}|\le t/|w|\}$,
which is contained in $S(I)$ 
for some arc $I$ with $|I|=8t$.
By \eqref{E:Wynn}, we therefore have
\[
\mu\Bigl(\{z\in\DD:|1-\overline{w}z|\le t\}\Bigr)
\le C\phi(8t),
\]
where $C$ is a constant independent of $w$ and $t$.
We thus obtain
\[
\int_\DD\log\Bigl(\frac{2}{|1-\overline{w}z|}\Bigr)\,d\mu(z)
\le \int_{t=0}^2 C\phi(8t)\frac{1}{t}\,dt
=\int_{s=0}^{16}C\frac{\phi(s)}{s}\,ds.
\]
By \eqref{E:phicond}, this last integral is finite.
This gives \eqref{E:Carllog}.
\end{proof}

\section{Proof of Theorem~\ref{T:Wynnsharp}}\label{S:pfthm2}

The proof of Theorem~\ref{T:Wynnsharp} is inspired by a construction
of Stegenga \cite[\S4]{St80}.
Translated into our context, his result is 
essentially the special case of Theorem~\ref{T:Wynnsharp}
in which $\phi(x):=1/\log(1/x)$.
The construction proceeds via a characterization
of Carleson measures for $\cD$, also due to Stegenga.
This characterization is expressed in terms of logarithmic capacity,
so we take a moment to define this notion and 
summarize those of its properties that we shall need.

Let $E$ be compact subset of $\TT$.
Its \emph{logarithmic capacity} $c(E)$ is defined
by the formula
\[
\frac{1}{c(E)}:=\inf_\nu\int_E\int_E\log\Bigl(\frac{2}{|1-\overline{w}z|}\Bigr)\,d\nu(z)\,d\nu(w),
\]
where the infimum is taken over all Borel probability 
measures $\nu$ on $E$. It could happen that the infimum is infinite,
in which case $c(E)=0$.
We shall need the following facts:
\begin{itemize}
\item Capacity is upper semicontinuous: if $E_n\downarrow E$ then $c(E_n)\downarrow c(E)$.
\item For arcs $I$, we have $c(I)\asymp 1/\log(1/|I|)$ as $|I|\to0$.
\item Let $E$ be a generalized Cantor set in the unit circle, 
formed by taking an arc of length $l_0$,
removing  an arc from its center to leave two arcs of length~$l_1$,
removing an arc from each of their centres to leave four arcs of length~$l_2$, 
and so on.
Then 
\begin{equation}
c(E)=0 \iff \sum_{n\ge0}2^{-n}\log(1/l_n)=\infty.
\end{equation}
\end{itemize}
For further information about logarithmic capacity,
we refer to \cite[Chs.~III \& IV]{Ca67}.

We now state Stegenga's characterization of Carleson measures \cite[Theorem~2.3]{St80}.

\begin{theorem}\label{T:Stegenga}
Let $\mu$ be a finite positive Borel measure on $\DD$.
Then $\mu$ is a Carleson measure for~$\cD$ if and only if
there exists a constant $A$ such that,
for every finite set of disjoint closed subarcs 
$I_1,\dots,I_n$ of $\TT$,
\begin{equation}\label{E:Stegenga}
\mu\Bigl(\cup_{k=1}^nS(I_k)\Bigr)
\le A c\Bigl(\cup_{k=1}^n I_k\Bigr).
\end{equation}
\end{theorem}

With this result under our belt, we can  prove Theorem~\ref{T:Wynnsharp}.

\begin{proof}[Proof of Theorem~\ref{T:Wynnsharp}]
We can suppose  that $\lim_{x\to0}\phi(x)=0$, otherwise condition \eqref{E:Wynn} is vacuous.
Multiplying $\phi$ by a constant, 
we may further suppose that $\phi(1)=1$.
Then, for each $n\ge0$, 
there exists $l_n\in(0,1]$ such that $\phi(l_n)=2^{-n}$.
Since $\phi(x)/x$ is strictly decreasing,  
$l_{n+1}<l_n/2$ for all~$n$.
Let $E$ be the associated generalized Cantor set, 
as described above.
Let $\sigma$ the corresponding Cantor--Lebesgue measure
(namely the probability measure on $E$ giving weight $2^{-n}$
to each of the $2^n$ arcs appearing at the $n$-th stage in the construction of $E$).
Let $(\delta_n)_{n\ge0}$ be any decreasing sequence in $(0,1)$.
Let $\mu_n$ be the measure on $\DD$ defined by
$\mu_n((1-\delta_n)B):=\sigma(B)$ 
(so $\mu_n$ is just $\sigma$, scaled to live on 
the slightly smaller circle $|z|=1-\delta_n$).
Finally, let $\mu:=\sum_{n\ge0}2^{-n}\mu_n$.
We claim that $\mu$ satisfies \eqref{E:Wynn},
and  that we may choose $(\delta_n)$
so that $\mu$ is not a Carleson measure for $\cD$.

Let us  show that $\mu$ satisfies \eqref{E:Wynn}.
Let $I$ be an arc with $|I|<l_0$.
Pick $n$ such that $l_{n+1}\le |I|< l_n$. 
The arcs appearing in the $n$-th stage of the construction of $E$ have length  $l_n$, 
so $I$ can meet at most two of them, 
whence $\sigma(I)\le 2.2^{-n}=4.2^{-(n+1)}=4\phi(l_{n+1})\le 4\phi(|I|)$.
It follows that 
\[
\mu(S(I))=\sum_{k\ge0}2^{-k}\mu_k(S(I))\le\sum_{k\ge0}2^{-k}\sigma(I)=2\sigma(I)\le 8\phi(|I|).
\]
This implies that \eqref{E:Wynn} holds.

Now we show that, if $(\delta_n)$ is chosen appropriately, then $\mu$ is not a Carleson measure for~$\cD$.
By \eqref{E:Wynnsharp}, 
the integral $\int_0^{l_0}(\phi(x)/x)\,dx$ diverges.
On the other hand, it is bounded above by
\[
\sum_{n\ge0}\phi(l_{n})\int_{l_{n+1}}^{l_n}\frac{dx}{x}
=\sum_{n\ge0}2^{-n}\log(l_{n}/l_{n+1})
\le\sum_{n\ge0}2^{-n}\log(1/l_{n+1}).
\]
Therefore $\sum_n2^{-n}\log(1/l_n)=\infty$.
As mentioned at the beginning of the section, 
this implies that $c(E)=0$.
Set $E_{\delta}:=\{\zeta\in\TT: d(\zeta,E)\le \delta\}$.
Since capacity is upper semicontinuous, 
we have $c(E_\delta)\to c(E)=0$ as $\delta\to0$.
Therefore, we may choose $\delta_n$ so that $c(E_{\delta_n})<3^{-n}$
for all $n$.
For each $n$, the set $E_{\delta_n}$  is a finite union of closed arcs, $I_1,\dots,I_k$, 
each of length at least $\delta_n$.
The sets $S(I_j)$ therefore all meet the circle
$|z|=1-\delta_n$.
It follows that
\[
\mu\bigl(\cup_{j=1}^kS(I_j)\bigr)
\ge 2^{-n}\mu_n\bigl(\cup_{j=1}^kS(I_j)\bigr)
= 2^{-n}\sigma\bigl(\cup_{j=1}^kI_j)
\ge 2^{-n}\sigma(E)=2^{-n}.
\]
On the other hand, 
\[
c(\cup_{j=1}^k I_j)=c(E_{\delta_n})<3^{-n}.
\]
Thus, for \eqref{E:Stegenga} to hold, we must have $2^{-n}<A3^{-n}$. 
Obviously, there is no constant $A$ such that this holds for all $n$,
so, by Theorem~\ref{T:Stegenga}, the measure  $\mu$ is not a Carleson measure for $\cD$.
\end{proof}

\section{Remarks about necessity}\label{S:Carlnec}

(i) As mentioned at the beginning of \S\ref{S:pfthm2}, for arcs $I$ we have
$c(I)\asymp 1/\log(1/|I|)$ as $|I|\to0$.
Thus, applying Theorem~\ref{T:Stegenga} with one arc, 
we see that a necessary condition for $\mu$ to be a Carleson measure for $\cD$
is that
\begin{equation}\label{E:Carlnec}
\mu(S(I))=O\Bigl(\Bigl(\log\frac{1}{|I|}\Bigr)^{-1}\Bigr)
\qquad(|I|\to0),
\end{equation}
as claimed in \eqref{E:CarlD}.

(ii) The sufficient condition \eqref{E:Carllog} is not necessary
for $\mu$ to be a Carleson measure for $\cD$. Indeed, 
consider be the measure $\mu$ on $[0,1)$ defined by
\[
\mu\Bigl([1-t,1)\Bigr)=\Bigl(\log\frac{2}{t}\Bigr)^{-1}
\qquad(0\le t<1).
\]
By Theorem~~\ref{T:Stegenga}, $\mu$ is Carleson for $\cD$.
On the other hand, using Fubini's theorem to integrate by parts, we have
\[
\int_{[0,1)}\log\frac{2}{1-wt}\,d\mu(t)
=\int_{x=1-w}^2\frac{dx}{x\log(2w/(w+x-1))}
\qquad(0<w<1),
\]
and the right-hand side clearly tends to infinity as $w\to1^-$,
so \eqref{E:Carllog} fails.

\bibliographystyle{amsplain}

\end{document}